\documentclass[11pt]{amsart}

\usepackage{url} 

\usepackage{amsmath}
\usepackage{amsfonts}
\usepackage{amsthm} 
\usepackage{enumitem} 
\usepackage{tikz-cd} 
\usepackage{hyperref} 
\usepackage{mathtools} 
\usepackage{amssymb} 
\usepackage{quiver} 
\usepackage{extarrows}

\theoremstyle{definition}
\newtheorem{thm}{Theorem}[section] 
\newtheorem*{thm*}{Theorem} 
\newtheorem{definition}[thm]{Definition} 
\newtheorem*{definition*}{Definition}
\newtheorem{lemma}[thm]{Lemma}
\newtheorem{example}[thm]{Example} 
\newtheorem{prop}[thm]{Proposition}
\newtheorem{coro}[thm]{Corollary}
\newtheorem{conv}[thm]{Convention}
\newtheorem{obs}[thm]{Observation}
\newtheorem*{notn}{Notation}

\usepackage{parskip} 
\DeclareMathOperator{\Top}{Top}
\DeclareMathOperator{\Colim}{colim}
\DeclareMathOperator{\Set}{Set}
\DeclareMathOperator{\Hom}{Hom}
\DeclareMathOperator{\Free}{Free}
\DeclareMathOperator{\id}{id}
\DeclareMathOperator{\pt}{pt}

\title{Elmendorf's Theorem for Diagrams}
\author{Hannah Housden}
\address{
Department of Mathematics,
Vanderbilt University,
Nashville, TN 37240 USA
}
\email{hannah.housden@vanderbilt.edu}

\begin{document}

\begin{abstract}
    The notion of a continuous $G$-action on a topological space readily generalizes to that of a continuous $D$-action, where $D$ is any small category. Dror Farjoun and Zabrodsky introduced a generalized notion of orbit, which is key to understanding spaces with continuous $D$-action. We give an overview of the theory of orbits and then prove a generalization of ``Elmendorf's Theorem,'' which roughly states that the homotopical data of of a $D$-space is precisely captured by the homotopical data of its orbits.
\end{abstract}

\thanks{This material is based upon work supported by the National Science Foundation under Grant No. DMS-1811189.}

\maketitle

\section{Introduction}
Equivariant homotopy theory is the study of topological spaces with the action of a (usually finite) group $G$. Any $G$-space $X$ automatically inherits an action via any subgroup $H \leq G$, and the corresponding fixed-point subspaces $X^H$ are key to understanding the structure of $X$. One concrete way to see this is via the celebrated result of ``Elmendorf's theorem,'' which was originally proven by Elmendorf \cite{Elmendorf}. The following is a reformulation due to Piacenza \cite[Theorem 6.3]{Piacenza}:

\begin{thm*}[Elmendorf's theorem, classical version]
    Let $G$ be a topological group, and let $\mathcal{O}_G$ be the category of $G$-orbits (that is, $G$-spaces of the form $G/H$). Then, there is a Quillen equivalence between $\Top^G$ and $\Top^{\mathcal{O}_G^{op}}$.
\end{thm*}

Our main result is a generalization of Elmendorf's theorem to ``$D$-spaces", functors from a small category $D$ to $\Top$. $D$-spaces generalize equivariant homotopy theory because every group can be viewed as a small category. For this reason, we'll refer to the action of $D$ as ``$D$-equivariance'' or ``diagram equivariance.'' In the 1980s, $D$-spaces were studied to prove categorical facts in homotopy theory. For instance, the version of Elmendorf's theorem above uses the category $D = \mathcal{O}_G^{op}$. In a similar vein, functors out of the poset $\mathbb{N}$ are often used to define spectra. $D$-spaces are also interesting objects in their own right, and we will give an overview of some common examples in Section \ref{section:basicnotions}. In the meantime, let us state our main theorem:

\begin{thm*}[Elmendorf's theorem, diagram-equivariant version]
    Let $D$ be a small\footnote{That is, a category with an actual set of objects where $\Hom(X,Y)$ is also always a set.} category, and let $\mathcal{O}_D$ be the category of $D$-orbits (that is, $D$-spaces $X$ where $\Colim(X)$ is terminal). Then, there is a Quillen equivalence between $\Top^D$ and $\Top^{\mathcal{O}_D^{op}}$.
\end{thm*}

The key definition that facilitates the study of $D$-spaces is the aforementioned notion of $D$-orbit, which is originally due to Dror Farjoun and Zabrodsky in \cite{DF}. In our version of Elmendorf's theorem, the model structures on $\Top^D$ and $\Top^{\mathcal{O}_D^{op}}$ are defined precisely so that they track the orbits of $D$. Namely:

\begin{definition*}
    In the \textit{projective model structure} on $\Top^{\mathcal{O}_D^{op}}$, weak equivalences (resp. fibrations) are those maps $\beta \colon R \rightarrow S$ such that, for each object $O \in \mathcal{O}_D$, $\beta_O$ is a weak equivalence (resp. fibration).
\end{definition*}

\begin{definition*}
    In the \textit{$\mathcal{O}_D$ model structure} on $\Top^D$ weak equivalences (resp. fibrations) are those maps $\alpha \colon X \rightarrow Y$ such that, for each object $O \in \mathcal{O}_D$, $\Top^D(O, \alpha)$ is a weak equivalence (resp. fibration).
\end{definition*}

Note that while $D$ and $\mathcal{O}_D^{op}$ are both categories, the model structures on $\Top^D$ and $\Top^{\mathcal{O}_D^{op}}$ are quite different. When Piacenza was proving his version of Elmendorf's theorem, he only needed the projective model structure because he was only considering categories of the form $\mathcal{O}_G^{op}$. In this paper, we're doing something quite different, which is to allow a small category $D$ to replace the group $G$. Thus, for our theorem to work, we create a model structure that directly generalizes the one Piacenza uses on $\Top^G$. In fact, because the model structures on $\Top^D$ and $\Top^{\mathcal{O}_D^{op}}$ are both based on orbits, we actually prove a more general result:

\begin{thm*}[Main Theorem]
        Let $D$ be a small category, let $\mathcal{F}$ be a collection of orbits containing all free $D$-orbits (those of the form $D(d,-)$ for some object $d \in D$), and let $\mathcal{O}_{\mathcal{F}}$ be the category of $D$-orbits in $\mathcal{F}$. Then, there is a Quillen equivalence between $\Top^D$ and $\Top^{\mathcal{O}_{\mathcal{F}}^{op}}$.
\end{thm*}

For our main theorem, we use model structures where the weak equivalences and fibrations only involve orbits in $\mathcal{F}$. In the case where $D$ is a group, this generalization was proven by Stephan \cite{Stephan}; our proof follows much of the same argument. This generalization is especially useful in the diagram-equivariant setting because $\mathcal{O}_D$ is often a large category, but only a small set of orbit types appear in any given $D$-space.

We prove our diagram-equivariant Elmendorf's theorem with an eye toward equivariant \textit{stable} homotopy theory. In the past decade, there have been approaches due to Barwick \cite{Barwick} and Guillou-May \cite{Guillou-May} defining equivariant spectra via ``spectral Mackey functors.'' These are functors whose domain is a modified version of $\mathcal{O}_G^{op}$. This author's PhD thesis \cite{HousdenDissertation} uses a version of spectral Mackey functors to define diagram-equivariant spectra. By contrast, older approaches (going as far back as Graeme Segal's paper \cite{SegalESHT} establishing equivariant stable homotopy theory) rely heavily on representation theory and the representation spheres $S^V$. The key limitation of the representation-based approach is that one can't easily define representation spheres when $G$ isn't a compact Lie group. Working with orbits directly allows one to sidestep this limitation. After all, the classical version of Elmendorf's theorem holds for general topological groups. For this reason, attempts to define equivariant spectra for groups that aren't compact Lie, such as $\mathbb{Z}$, have taken a more orbit-centric approach. While this paper focuses entirely on $D$-spaces, our envisioned applications are these orbit-centric versions of equivariant spectra.

\subsection*{Outline} The paper is organized as follows: Section \ref{section:basicnotions} lays out the first properties, related definitions, and examples of $D$-spaces. Section \ref{section:DProperties} explores the categorical properties of $\Top^D$. Section \ref{section:orbits} explains how $D$-orbits perform the same role in $D$-equivariant homotopy theory that subgroups $H \leq G$ do in $G$-equivariant homotopy theory. (These notions are indeed compatible when $D$ is a group.) Section \ref{section:invariants} discusses the homotopy groups of $D$-spaces and how they are naturally indexed by the category of $D$-orbits. Section \ref{section:complexes} explores the notions of $D$-CW-complexes and $D$-cell complexes, which are also due to Dror Farjoun and Zabrodsky. Section \ref{section:models} contains the model structures on $\Top^D$ and $\Top^{\mathcal{O}_{\mathcal{F}}^{op}}$, and Section \ref{section:Elmendorf} uses them to prove our generalized Elmendorf's theorem.

\subsection*{Acknowledgements} 

This contents of this paper formed the first half of my PhD thesis, and I'm grateful to my advisor, Mike Hill, for the many hours of discussion that led to these results. I'd also like to thank Anna Marie Bohmann for her many helpful comments on previous drafts of this paper.

\section{Basic Notions}\label{section:basicnotions}
Equivariance is often encoded as a continuous group action on a topological space. To view this more categorically, we recall that a group can be defined as a (small) category with one object where every morphism is an isomorphism. In this context, a space with a $G$-action is just a functor $X \colon  G \rightarrow \Top$, where $G$ is regarded as a category.

This definition would work just as well if $G$ were any category, which leads us to:

    \definition Given a small category $D$, a \textit{$D$-space} is a functor $X \colon  D \rightarrow \Top$. A morphism of $D$-spaces $\alpha \colon  X \rightarrow Y$ is a natural transformation.

This categorification works just as well for other kinds of ``objects with group action.'' For instance, a \textit{$D$-set} is a functor $X \colon  D \rightarrow \Set$ and a \textit{D-representation} is a functor $X \colon  D \rightarrow \mathrm{Vect}_k$ where $\mathrm{Vect}_k$ is the category of vector spaces over a fixed field, $k$. As with $D$-spaces, morphisms are natural transformations. We now repeat our key example, orbits, which generalize $G$-sets of the form $G/H$:

    \definition[{\cite[Definition 1.1]{DF}}] Given a category $D$, we say that a $D$-space $X\colon  D \rightarrow \Top$ is an \textit{orbit} if the colimit of $X$ is terminal (that is, a one-point space).
 
We are especially interested in those orbits that we can analyze via the Yoneda lemma:
   
    \definition[{\cite[Definition 2.2]{DFAlone}}] Given a category $D$ and object $d$, the \textit{free orbit of $d$}, $F^d$, is the representable $D$-set (and discrete $D$-space) $D(d,-)$.
    
    \prop For any object $d \in D$, the free orbit $D(d,-)$ is indeed an orbit.
        \begin{proof}
            For any morphism $f\colon d \rightarrow d'$ with source $d$, $\id_d \circ f = f$. Thus, $D(d,f)\colon D(d,d) \rightarrow D(d,d')$ sends $\id_d$ to $f$, so $\id_d$ and $f$ are glued together in $\Colim(D(d,-))$. Because this holds for all $f$ with source $d$ (that is, all elements of $\amalg_{d'}D(d,d')$), the colimit of $D(d,-)$ is a one-point set, meaning $D(d,-)$ is an orbit.
        \end{proof}

When $D$ is a group, the sole free orbit is isomorphic to $D/\{e\}$ due to $D$ only having one object. As for the other orbits, one can indeed check that $D$ being a group implies that any $D$-orbit $O$ is isomorphic to $D/H$ for some subgroup $H$. However, things generally get much more exciting for non-group categories, even very simple ones. For example:

    \definition\label{Jdef} Let $\mathbb{J} = s \xrightarrow{f} t$ be the category with two objects (``s'' for ``source'' and ``t'' for ``target'') and one non-identity morphism $f\colon  x \rightarrow y$.
    
    \prop\label{LargeJOrbits} There is an equivalence of categories between the category of $\mathbb{J}$-orbits and $\Top$.
    
        \begin{proof}
        By the usual construction of colimits in $\Top$, the colimit of $X\colon \mathbb{J} \rightarrow \Top$ is given as the quotient $(X_s \amalg X_t)/ \sim$, where $\sim$ relates each $x_s \in X_s$ to $f(x_s) \in X_t$. Note that each $x_s \in X_s$ is related to exactly one point in $X_t$ because $f$ is a function. Thus, it's not possible for two distinct points in $X_t$ to become glued together in the quotient, so any orbit $X$ must have a singleton $X_t$. This then forces all points in $X_s$ to be glued to the unique $x_t \in X_t$. Hence, a $\mathbb{J}$-orbit is precisely a $\mathbb{J}$-space $X\colon \mathbb{J} \rightarrow \Top$ such that $X_t$ has a single point. From this we observe that any continuous map $\alpha_s\colon  X_s \rightarrow Y_s$ will yield a commutative square

        \[\begin{tikzcd}
        	{X_s} && {Y_s} \\
        	\\
        	{X_t} && {Y_t}
        	\arrow["{\alpha_S}", from=1-1, to=1-3]
        	\arrow["{\alpha_t}", from=3-1, to=3-3]
        	\arrow["{X_f}"', from=1-1, to=3-1]
        	\arrow["{Y_f}"', from=1-3, to=3-3]
        \end{tikzcd}\]
        
        \noindent whenever $Y$ is an orbit. $\alpha$ is then uniquely determined. In other words, the functor $\Top^{\mathbb{J}} \rightarrow \Top$ taking $X$ to $X_s$ is full, faithful, and essentially surjective (i.e., an equivalence).
        \end{proof}
        
We'll use finite $\mathbb{J}$-orbits as a frequent source of examples, so it will be helpful to have simple notation for them: 

    \notn For any $n \in \mathbb{N}$, $[n]$ is the $\mathbb{J}$-orbit $\{ 0, \dots n-1 \} \rightarrow \{ \pt \}$.
    
    \obs The free orbits of $\mathbb{J}$ are $[0] \cong F^t = \mathbb{J}(t,-)$ and $[1] \cong F^s = \mathbb{J}(s,-)$.

For a general small category $D$, we can still have a lot of orbits to work with. In most situations, we only need to consider the orbits that actually appear in our $D$-space $X$. Let us introduce some vocabulary to describe these orbits.

    \conv \label{constantconvention} For any topological space $A$, we will abuse notation and also refer to its corresponding constant $D$-space as $A$. (Thus, the statement $A_d = A$ is perfectly valid.)

    \definition[{\cite[Definition 2.2]{DF}}] \label{orbitofx}  Given a $D$-space $X$ and a point $x_d$ of $X_d$ for some object $d \in D$, the \textit{orbit of $x_d$}, $O_{x_d}$, is the $D$-space of points in $X$ that get glued to $x_d$ in $\Colim(X)$. This can be identified with the following pullback: (The spaces on the bottom row are viewed as constant $D$-spaces, following Convention \ref{constantconvention}.)

    \[\begin{tikzcd}
    	{O_{x_d}} && X \\
    	\\
    	{\{x_d\}} && {\Colim(X)}
    	\arrow[from=1-3, to=3-3]
    	\arrow[hook, from=3-1, to=3-3]
    	\arrow[from=1-1, to=3-1]
    	\arrow[hook, from=1-1, to=1-3]
    	\arrow["\lrcorner"{anchor=center, pos=0.125}, draw=none, from=1-1, to=3-3]
    \end{tikzcd}\]

Orbit types allow us to classify discrete $D$-spaces (that is, $D$-sets) as follows:

    \prop Given a small category $D$, any $D$-set $X$ can be decomposed as the coproduct of its (necessarily discrete) $D$-orbits. Furthermore, this decomposition is unique up to reordering and isomorphic replacement of the factors.
    
        \begin{proof}
            Given a point $x_d$ of $X$, each point $x \in O_{x_d}$ is by definition glued to $x_d$ in $\Colim(X)$. Thus, each point of $X$ is in precisely one orbit. This gives a disjoint union decomposition of $X$, which is precisely the claimed coproduct structure.
            
            To see uniqueness, suppose $X$ is isomorphic to both $\coprod_{i \in I} O_i$ and $\coprod_{j \in J} \widetilde{O_j}$, where each $O_i$ and $\widetilde{O_j}$ is a $D$-orbit. The fact that these are both decompositions of $X$ give us an isomorphism \[f\colon  \coprod_{i \in I} O_i \rightarrow \coprod_{j \in J} \widetilde{O_j}.\] Because isomorphisms preserve colimits, $f$ induces an isomorphism of sets \[g\colon \Colim(\coprod_{i \in I} O_i) \rightarrow \Colim(\coprod_{j \in J} \widetilde{O_j}).\] Since colimits commute with coproducts, $g$ can instead be viewed as a bijective function \[g\colon \coprod_{i \in I} \Colim(O_I) \rightarrow \coprod_{j \in J} \Colim(\widetilde{O_j}).\] But each $\Colim(O_i)$ and $\Colim(\widetilde{O_j})$ is a one-point set, so $g$ corresponds to a bijection from $I$ to $J$, which we will call $h$. The fact that $h$ sends $i$ to $h(i)$ means that $f$ sends points in $O_i$ to points in $\widetilde{O}_{h(i)}$. Because $h$ is a bijection, only the points in $O_i$ can be sent to $\widetilde{O}_{h(i)}$. Thus, since $f$ is an isomorphism, the restriction of $f$ to $O_i \rightarrow \widetilde{O}_{h(i)}$ must also be a bijection and hence an isomorphism. This shows the desired uniqueness.
        \end{proof}

\section{Properties of $\mathrm{Top}^D$} \label{section:DProperties}
Let us now explore the categorical properties of $\Top^D$:

    \prop For any small category $D$, $\Top^D$ has all small limits and colimits.
        \begin{proof}
            This is immediate from the fact that $\Top$ has all small limits and colimits and the fact that limits and colimits in a functor category can be computed objectwise. 
        \end{proof}

Next, we'd like see that $\Top^D$ is nicely enriched in $\Top$. What follows is largely a recap of \cite[Section 3]{DF}.

    \definition For any $D$-spaces $X$ and $Y$, we topologize $D(X,Y)$ with the subspace topology from its inclusion into $\Top(\amalg_d X_d, \amalg_d Y_d)$.
    
    \coro We can view $\Top^D$ as enriched in $\Top$.
    
However, we can also view $D(X,Y)$ as a $D$-space:

    \definition \cite[Proposition 2.17]{DFAlone} For any $D$-spaces $X$ and $Y$, we can view $D(X,Y)$ as a $D$-space where $D(X,Y)_d = D(X \times F^d,Y)$. For any morphism $f:d \rightarrow d'$ in $D$, \[D(X,Y)_f: D(X \times F^d, Y) \rightarrow D(X \times F^{d'}, Y)\] is induced by the natural map \[D(f,-):F^{d'} = D(d',-) \rightarrow D(d,-) = F^d.\] This construction makes $\Top^D$ enriched in itself. Many times, we'll use the following notation for $D(X,Y)$:

    \notn For $D$-spaces $X$ and $Y$, $Y^X \coloneqq \Top^D(X,Y)$. Whether we wish for $Y^X$ to be a space or a $D$-space will depend on context.

Most commonly, we'll be applying this notation to the context of the representable ``fixed point'' functor $(-)^Y$. Let's list a few of this functor's properties: 

    \prop \label{Ofixedlimits} For any $D$-space $X$, the functor $(-)^X$ (valued in either $\Top$ or $\Top^D$) preserves limits.
        \begin{proof}
            This is immediate from the fact that (enriched) representable functors preserve limits. 
        \end{proof}
        
    \prop \label{OfixedPushout} For any $D$-orbit $O$, the functor $(-)^O$ (valued in either $\Top$ or $\Top^D$) preserves pushouts and coproducts.
        \begin{proof}
            We begin with pushouts: Consider any pushout $Y \sqcup_W Z$ of $D$-spaces. Because ${\textrm{colim}(O) = \{o\}}$ is a one-point space, a map $f:O \rightarrow Y \sqcup_W Z$ is factored by a map $O \rightarrow Y$ or $O \rightarrow Z$ based on whether a given point $f(o)$ lands in $\Colim(Y) \subseteq \Colim(Y \sqcup_W Z)$ or in $\Colim(Z) \subseteq \Colim(Y \sqcup_W Z)$. If $f(o)$ lands in both $\Colim(Y)$ and $\Colim(Z)$, then $O \rightarrow Y$ and $O \rightarrow Z$ are both factored by a map $O \rightarrow X$. In other words, $(Y \sqcup_W Z)^O$ has the universal property of a pushout of $Y^O \leftarrow W^O \rightarrow Z^O$, so $(Y \sqcup_W Z)^O$ and $Y^O \sqcup_{W^O} X^O$ are naturally isomorphic as spaces. In other words, $(-)^O$ preserves pushouts, at least when it's valued in $\Top$.
            
            Similarly, for any coproduct $\amalg_{i \in I} Y_i$ of $D$-spaces, a map \[O \rightarrow \amalg_{i \in I} Y_i\] is factored based on which $\textrm{colim}_{d \in D}(Y_i) \subseteq \textrm{colim}_{d \in D}(\amalg_{i \in I} Y_i)$ that $f(o)$ lands in. Thus, $(\amalg_{i \in I} Y_i)^O$ has the universal property of $\amalg_{i \in I} Y^O_i$, at least when $(-)^O$ is valued in $\Top$.
            
            To get the version valued in $\Top^D$, recall that $X^O_d$ is the space $\Top^D(F^d,X^O)$. Since $F^d$ is an orbit, we conclude the pushouts and coproducts are preserved at each object. Since equivariant maps that have objectwise-isomorphisms are themselves isomorphisms, we conclude that the $\Top^D$-valued functor $( - )^O$ preserves pushouts and coproducts.
        \end{proof}

We can say more about this enrichment if we take $\Top$ to be the category of compactly generated weak Hausdorff spaces. This condition implies that, for any spaces $A, B, C$, we have an isomorphism of spaces \[\Top(A \times B, C) \cong \Top(A,\Top(B,C)).\]

As long as we have this condition, we get the following:

    \prop $\Top^D$ is a closed symmetric monoidal category, with monoidal structure given by objectwise Cartesian product.
        \begin{proof}
            We just need to confirm that there is a natural isomorphism of sets \[\Top^D(X \times Y,Z) \cong \Top^D(X,\Top^D(Y,Z)).\] 
            Given a natural transformation \[\alpha:X \times Y \rightarrow Z,\] we get a natural transformation \[\beta: X \rightarrow \Top^D(Y,Z),\] where \[\beta_d:X_d \rightarrow \Top^D(Y,Z)_d = \Top^D(Y \times F^d,Z)\] is defined by \[ [\beta_d(x_d)](y_{d'},f) = \alpha(f(x_d),y_{d'}).\] (Here, $f$ is a generic element of $F^d(d')$, meaning it's a morphism $f:d \rightarrow d'$.) The fact that we're working with compactly generated weak Hausdorff spaces ensures that each $\beta_d$ is continuous. We can recover $\alpha$ from $\beta$ by setting \[ \alpha_d(x_d,y_d) = [\beta_d(x_d)](y_d,\id_d).\] Again, the fact that we're working with compactly generated weak Hausdorff spaces ensures that each $\alpha_d$ is continuous.
        \end{proof}
        
If we apply the same argument to \textit{pointed} compactly generated weak Hausdorff spaces, using the isomorphism \[\Top_{\bullet}(D \wedge E,F) \cong \Top_{\bullet}(D,\Top(E,F)),\] for any pointed spaces $D,E,$ and $F$, we get:

    \coro \label{pointedhomtensor} $\Top^D_{\bullet}$ is closed symmetric monoidal category, with monoidal structure given by objectwise smash product.

\section{Orbits vs. Subgroups} \label{section:orbits}
In the group case, keeping track of orbits of $G$ is essentially the same task as keeping track of subgroups of $G$. One way of making this precise is the following:

    \prop The category of $G$-orbits with $G$-equivariant maps, $\mathcal{O}_G$, is equivalent to the category of subgroups of $G$ with inclusions and conjugations for morphisms. 
    
For general small categories $D$, the natural generalization of this proposition that uses ``subcategory'' instead of ``subgroup'' is not even remotely true. For $\mathbb{J}$ in Definition \ref{Jdef}, we saw in Proposition \ref{LargeJOrbits} that the orbit category was equivalent to $\Top$, but we can see there are only finitely many subcategories! Thankfully, we don't need to use subgroups to capture the equivariant structure, and our story can be explained purely in terms of orbits. Let's explore how to go about this; for the group case, this will involve translating notions that use the subgroup $H \leq G$ into the language of $G$-orbits $G/H$. But first, we'll need a definition:

    \definition Given a $D$-set $T\colon D \rightarrow \Set$, its \textit{translation category}, $B_D(T)$, is the category with objects given by elements of $\amalg_{d \in D} T_d$ and has morphism-sets defined by \[B_D(T)(a,b) = \{f \in D \mid T_f(a)=b \}.\]
    
    \prop This construction is functorial in $T$. 
    
    \example When we pick $D$ to be a group and $T$ to be some orbit $D/H$, we get what is usually called the \textit{translation groupoid}. This translation groupoid, $B_D(D/H)$, is in fact equivalent (in the categorical sense) to $H$. Thus, the functor categories $\Top^H$ and $\Top^{B_D(D/H)}$ are categorically equivalent. Hence, we can talk about ``restricted'' action of subgroups purely in terms of orbits: while any $G$-space $X$ has a ``restricted'' $H$ action given by the inclusion $H \leq G$, $X$ also gives rise to a $B_G(G/H)$-space that encodes the same data. We can use a similar technique to discuss the fixed-point spaces of $X$: the set $X^H$ of points in $X$ that are fixed by the action of $H$ can be identified with $\Top^G(G/H, X)$.

\section{The Homotopy Theory of $D$-Spaces}\label{section:invariants}
To use algebraic invariants for $D$-spaces, we need to choose how much of the $D$-equivariant structure to capture. In classical (non-equivariant) homotopy theory, the $n$th homotopy groups of a pointed space $X$, $\pi_n(X)$, is usually viewed as consisting of the homotopy classes of pointed maps from $S^n$ to $X$. If we want to do this equivariantly, $X$ will have an action attached to it, and we need to decide what action $S^n$ has. If we give $S^n$ the constant ``trivial'' action, (that is, viewing $S^n$ as a functor $D \rightarrow \Top_{\bullet}$ that sends every object to the sphere $S^n$ and every morphism to the identity map on $S^n$) then we're extremely limited in the power of our invariants. For instance:

    \example Let $C_2$ denote the group of order $2$, and let $V$ be the $m$-dimensional orthogonal $C_2$-representation where the non-identity morphism acts via multiplication by $-1$. For any $n$, there is only one pointed $C_2$-equivariant map from $S^n$ to $S^V$. 
    
In other words, pointed $C_2$-equivariant maps from various $S^n$ can't distinguish between the $S^V$ above and a point. If we built a homotopy invariant out of the homotopy classes of such maps, we'd have a very weak invariant. A common solution to this problem is to instead build an invariant from the data of $\pi_n(X^H)$ for all subgroups $H \leq G$. This is the approach we'll adapt, using the fact that $X^H \cong \Top^G(G/H,X)$ to make an orbit-theoretic statement. But first, let's note that the weakness of only considering spheres with trivial actions isn't unique to the group-equivariant case:

    \example Let $X$ be the $\mathbb{J}$-space where $X_s = \{ \pt \}$ and $X_t = S^m$. For any $n$, there is only one $\mathbb{J}$-equivariant map from $S^n$ to $X$. (Here, we're following Convention \ref{constantconvention} and treating $S^n$ as a constant $\mathbb{J}$-space, which is the same as saying that $S^n$ has the ``trivial'' action.)
    
In both examples, the issue is orbit types: an equivariant map can only send points of orbit type $O_1$ to points of orbit type $O_2$ if there's a $D$-equivariant map from $O_1$ to $O_2$. When we used $S^n$ with the trivial action, there was only one orbit type represented,\footnote{This is true for all connected categories. In general, a constant $D$-space has orbit types precisely corresponding to the connected components of $D$.} that orbit being $C_2/C_2$ in the first example and $[1]$ in the second example. However, there were other orbit types present in the codomain, namely $C_2/\{e\}$ and $[0]$, respectively.

We can capture the homotopical data for these other orbit types by replacing $S^n$ with the ``free'' space $S^n \wedge O_+$. ($O_+$ is the pointed $D$-space obtained from $O$ by adding a disjoint base point at every object.) Aside from potentially the base points, every point in $S^n \wedge O_+$ has orbit type $O$. By Corollary \ref{pointedhomtensor}, an equivariant map from $S^n \wedge O_+$ to $X$ is equivalent to the data of an equivariant map from $S^n$ to $X^O$.

Thus, instead of having a single $n$-th homotopy group, we have one for each orbit. By the representability of homotopy groups, these can be arranged into a functor:

    \definition Let $\mathcal{O}_D$ be the category of $D$-orbits with $D$-equivariant maps. Given a pointed $D$-space $X$, its \textit{$n$-th equivariant homotopy group functor} is a contravariant functor $\pi_n^*(X): \mathcal{O}_D^{op} \rightarrow Grp$ given by \[\pi_n^O(X) = [S^n \wedge O_+, X]^D \cong [S^n, X^O]^D.\] 
    
Here, $[ - , - ]^D$ denotes the set of $D$-equivariant homotopy classes of maps. $\pi_n^O(X)$ is indeed a group when $n \geq 1$ because the constant space $S^n$ is automatically a cogroup object in the category of $D$-spaces with $D$-homotopy classes of maps. Let's now explore this invariant with a few examples for $\mathbb{J}$: 
    
    \example Let $X$ be a constant $\mathbb{J}$-space. Then, $\pi_n^*(X)$ is the constant functor $\pi_n(X_t)$ (or equivalently, $\pi_n(X_s)$).
    
        \begin{proof}
        Consider any morphism $\alpha: (S^n \wedge O_+) \rightarrow X$:

        \[\begin{tikzcd}
        	{(S^n \wedge O_+)_s} && {X_s} \\
        	{} \\
        	{(S^n \wedge O_+)_t} && {X_t}
        	\arrow["{X_f}"', from=1-3, to=3-3]
        	\arrow["{(S^n \wedge O_+)_f}"', from=1-1, to=3-1]
        	\arrow["{\alpha_s}", from=1-1, to=1-3]
        	\arrow["{\alpha_t}", from=3-1, to=3-3]
        \end{tikzcd}\]
            
        Since $X_f$ is an identity morphism, $\alpha_s$ is uniquely determined as the composite $\alpha_t \circ (S^n \wedge O_+)_f$. Thus, the homotopy classes of $\mathbb{J}$-equivariant maps from $S^n \wedge O_+$ to $X$ can be identified with the non-equivariant homotopy classes of maps from $(S^n \wedge O_+)_t$ to $X_t$. Now recall that when $O$ is a $\mathbb{J}$-orbit, $O_t$ is a one-point space. Thus, $(S^n \wedge O_+)_t \cong S^n$, so $\pi_n^O(X) \cong \pi_n(X_t)$. Because this identification can be made compatibly for each orbit, we conclude that $\pi_n^*(X)$ is a constant functor.
        \end{proof}
        
In the above example, we didn't get any interesting orbit data. This was just because $X$ only had one orbit type. To see a more general behavior, let's revisit the case where $X_s = \{ \pt \}$ and $X_t = S^m$:

    \example Let $X$ be the $\mathbb{J}$-space with $X_s = \{ \pt \}$ and $X_t = S^m$. Then, $\pi_n^{[0]}(X) = \pi_n(S^m)$ and $\pi_n^O(X) = 0$ for all other orbits $O$. This uniquely determines the structure maps.

        \begin{proof}
            We will use the notation of the previous example. Since $X_s$ consists of a single point, the composite $X_f \circ \alpha_s$ must send all points of $(S^n \wedge O_+)_s$ to the base point of $X_t$. When $O$ is not the orbit $[0]$, the map $(S^n \wedge O_+)_f$ is surjective, which means $\alpha_t$ must send all points of $(S^n \wedge O_+)_t$ to the base point of $X_t$. Hence, $\pi_n^O(X) = 0$. When $O$ is the orbit $[0]$, we know that $(S^n \wedge O_+)_s \cong \{ \pt \}$ and $(S^n \wedge O_+)_t \cong S^n$. This means $\alpha_t$ can be any map from $S^n$ to $S^m$, so $\pi_n^{[0]}(X) \cong \pi_n(S^m)$. 
        \end{proof}
    
The next example illustrates that we care about the structure maps of $\pi_n^*(X)$, not just its evaluation on objects.

    \example Let $X$ be the $\mathbb{J}$-space with $X_s=S^1$ and $X_t=S^1$, but where $X_f$ is the ``double counter-clockwise winding'' map, hereafter denoted ``$2$.'' (If one views $S^1$ as the unit sphere in $\mathbb{C}$, this is the map given by $z \mapsto z^2$.) Then, $\pi_n^O(X) = \pi_n(S^1)$ for all orbits $O$. For any $\mathbb{J}$-equivariant map $g:O_1 \rightarrow O_2$, $\pi_n^g(X)$ is multiplication by $2$ when $O_1$ is the orbit $[0]$ and $O_2$ is a non-$[0]$ orbit; otherwise, $\pi_n^g(X)$ is the identity map. 
    
        \begin{proof}
            For any morphism ${\alpha: (S^n \wedge O_+) \rightarrow X}$, we have the commutative square

            \[\begin{tikzcd}
            	{(S^n \wedge O_+)_s} && {S^1} \\
            	{} \\
            	{(S^n \wedge O_+)_t} && {S^1}
            	\arrow["2"', from=1-3, to=3-3]
            	\arrow["{(S^n \wedge O_+)_f}"', from=1-1, to=3-1]
            	\arrow["{\alpha_s}", from=1-1, to=1-3]
            	\arrow["{\alpha_t}", from=3-1, to=3-3]
            \end{tikzcd}\]
        
            Note that $(S^n \wedge O_+)_s  \cong \bigvee_{i \in O_s} S^n$ and $(S^n \wedge O_+)_t  \cong S^n$. Thus, when $n \geq 2$, any map $\alpha_t$ is nullhomotpic. Furthermore, we can lift any nullhomotopy on $\alpha_t$ to a compatible one on $\alpha_s$ because $(S^n \wedge O_+)_f$ is an isomorphism on each of the wedge factors. Hence, $\pi_n^O(X) = 0$ when $n \geq 2$, so we only need to consider the $n=1$ case.
            
            Classically, we know that any pointed map from $\bigvee_{i \in O_s} S^1$ to $S^1$ has a unique pointed lift via the map $X_f = 2$. In particular, there must be only one lift of $\alpha_t \circ (S^1 \wedge O_+)_f$. Let's compare $\alpha_s$ with another such lift: 
            
            When $O$ is not the orbit $[0]$, there exists maps $h:O_t \rightarrow O_s$ such that $O_f \circ h = \id_{O_t}$. Thus, we have a map $\widetilde{\alpha_t}: (S^1 \wedge O_+)_t \rightarrow S^1$ given by $\widetilde{\alpha_t}$ = $\alpha_s \circ (S^1 \wedge h)$. However, we can consider the diagram

            \[\begin{tikzcd}
            	{(S^1  \wedge O_+)_t} \\
            	\\
            	& {(S^1 \wedge O_+)_s} && {S^1} \\
            	& {} \\
            	& {(S^1 \wedge O_+)_t} && {S^1}
            	\arrow["2", from=3-4, to=5-4]
            	\arrow["{(S^1 \wedge O_+)_f}", from=3-2, to=5-2]
            	\arrow["{\alpha_s}", from=3-2, to=3-4]
            	\arrow["{\alpha_t}", from=5-2, to=5-4]
            	\arrow["h", from=1-1, to=3-2]
            	\arrow["{\widetilde{\alpha_t}}", dashed, from=1-1, to=3-4]
            	\arrow["{\id_{(S^1 \wedge O_+)_t}}"', from=1-1, to=5-2]
            \end{tikzcd}\]
            
            and compute that $2 \circ \widetilde{\alpha_t} = 2 \circ \alpha_s \circ h = \alpha_t \circ (S^1 \wedge O_+)_f \circ h = \alpha_t.$
            
            Thus, $\widetilde{\alpha_t} \circ (S^1 \wedge O_+)_f$ is a pointed lift of $\alpha_t \circ (S^1 \wedge O_+)_f$. By the uniqueness of pointed lifts, this means $\alpha_s = \widetilde{\alpha_t} \circ (S^1 \wedge O_+)_f$. In other words, we have a commutative diagram:

            \[\begin{tikzcd}
            	{(S^1 \wedge O_+)_s} && {S^1} \\
            	{} \\
            	{(S^1 \wedge O_+)_t} && {S^1}
            	\arrow["2"', from=1-3, to=3-3]
            	\arrow["{(S^1 \wedge O_+)_f}"', from=1-1, to=3-1]
            	\arrow["{\alpha_s}", from=1-1, to=1-3]
            	\arrow["{\alpha_t}", from=3-1, to=3-3]
            	\arrow["{\tilde{\alpha_t}}", dashed, from=3-1, to=1-3]
            \end{tikzcd}\]
            
            Assuming still that $O \neq [0]$, observe that a $\mathbb{J}$-equivariant map from $(S^1 \wedge O_+)$ to $X$ thus uniquely determines a map from the constant $\mathbb{J}$-space $(S^1 \wedge O_+)_t$ to $X$, and vice versa. We could repeat the same argument. replacing $(S^1 \wedge O_+)$ with $(S^1 \wedge O_+) \times I$, to get a lifting of $D$-equivariant homotopies. We also note that the constant $\mathbb{J}$-space $(S^1 \wedge O_+)_t$ is isomorphic to $(S^1 \wedge [1]_+)$. Thus, when $O \neq [0]$, we compute \[\pi_1^O(X) = [S^1 \wedge O_+, X]^{\mathbb{J}} \cong [S^1 \wedge [1], X]^{\mathbb{J}} = [S^1, X_s] = \pi_1(S^1).\] These isomorphisms specify that the structure maps $\pi_1^g(X)$ are isomorphisms for $g:O_1 \rightarrow O_2$ when neither $O_1$ nor $O_2$ are $[0]$. 
            
            When $O = [0]$, we compute \[\pi_1^{[0]}(X) = [S^1 \wedge [0]_+, X]^{\mathbb{J}} \cong [S^1,X^{[0]}] \cong [S^1, X_t] = [S^1,S^1] = \pi_1(S^1).\] Now, we just need to determine the unresolved structure maps. Since $\pi_1^g(X)$ is an isomorphism when $O_1$ and $O_2$ are not $[0]$, we only need to consider the unique map $j:[0] \rightarrow [1]$. (All of the other unresolved maps are obtained by composing this map with some already-known isomorphism.)
            
            Recall again that $[0]$ and $[1]$ are isomorphic to the free orbits $\mathbb{J}(t,-)$ and $\mathbb{J}(s,-)$, respectively. Under this identification, $j:[0] \rightarrow [1]$ becomes $\mathbb{J}(f,-)$. This means the structure map from \[\pi_1^{[1]}(X) = [S^1 \wedge [1]_+,X]^{\mathbb{J}} \cong [S^1, X_s] = \pi_1(X_s) \] to \[\pi_1^{[1]}(X) = [S^1 \wedge [0]_+,X]^{\mathbb{J}} \cong [S^1, X_t] = \pi_1(X_t) \] is given by $\pi_1(X_f)$, which is multiplication by 2.
            
        \end{proof}

Comparing this example with the preceding one about constant $\mathbb{J}$-spaces shows why we needed to have structure maps: without the maps, the $X$ above would have been indistinguishable from the constant $\mathbb{J}$-space $S^1$. For our last two examples, let's see how orbits other than $[0]$ and $[1]$, the free obits, can provide useful data: 

    \example Let $X$ be the pointed $\mathbb{J}$-space where $X_s = S^m$, where $X_t = S^{\infty} = \bigcup_{n \in \mathbb{N}} S^n$, and where $X_f$ is the inclusion of $S^m$ into $S^{\infty}$. Then, $\pi_n^{[0]}(X) = 0$, while $\pi_n^O(X) = \pi_n(X_s)$ for all other obits. Given any $g: O_1 \rightarrow O_2$ where $O_1$ and $O_2$ are not $[0]$, $\pi_n^g(X)$ is the identity map.
    
        \begin{proof}
            For $O \neq [0]$, $(S^n \wedge O_+)_f$ and $(S^n \wedge O_+ \times I)_f$ are surjective, so any map (or homotopy of maps) from $S^n \wedge O_+$ to $X$ is factored by the inclusion of the constant $\mathbb{J}$-space $S^m$ into $X$. Thus, for $O \neq [0]$ (and the structure maps between such $O$), $\pi_n^O(X)$ agrees with $\pi_n^O(S^m) \cong \pi_n(S^m)$.
        \end{proof} 
        
We contrast this with the following:

    \example Let $Y$ be the pointed $\mathbb{J}$-space with $Y_s = S^m$ and $Y_t = \{ \pt \}$. Then, $\pi_n^{[i]}(Y) \cong \prod_{z \in [i]_s} \pi_n(S^m)$, and $g:[j] \rightarrow [i]$ acts by sending $(z_1, \dots, z_i)$ to $(z_{g(1)}, \dots , z_{g(j)})$.
    
        \begin{proof}
             Because $Y_t$ is terminal, the data of a map $\alpha$ from $S^n \wedge [i]_+$ to $Y$ is the same as the data of $\alpha_s: (S^n \wedge [i]_+)_s \rightarrow Y_s$. Since $(S^n \wedge [i]_+)_s$ is homeomorphic to the $i$-fold wedge product $S^n \vee \dots \vee S^n$, we have that \[\pi_n^{[i]}(Y) \cong [S^n \vee \dots \vee S^n, S^m] \cong \prod_{z \in [i]_s} \pi_n(S^m),\] where the last isomorphism follows from the fact that $\vee$ is the coproduct in the category of pointed topological spaces with homotopy classes of maps. Our description of $\pi_n^g(Y)$ then follows from chasing through the two isomorphisms.
        \end{proof}

In these last two examples, $\pi_n^{[0]}(X)$, $\pi_n^{[1]}(X)$, and $\pi_n^{[0] \rightarrow [1]}(X)$ agree with their counterparts for $Y$. Only by using the other orbits can we homotopically distinguish between $X$ and $Y$. This is desirable because while $X_s \simeq Y_s$ and $X_t \simeq Y_t$ are homotopy equivalent \textit{as spaces}, $X$ and $Y$ are not homotopic as $\mathbb{J}$-spaces. This is precisely analogous to distinguishing between $G$-spaces whose underlying spaces are homotopy equivalent but which are not equivariantly homotopy equivalent.

\section{$D$-CW-Complexes and $D$-cell complexes}\label{section:complexes}

As in the non-equivariant case, we have a notion of CW-complexes, objects that are completely described by homotopy group functors. The material here largely follows the original exposition given by Dror Fajoun and Zabrodsky, with some more modern updates. We will need $D$-cell complexes for our proof of the diagram-equivariant Elmendorf's theorem because they are used to build the cofibrant objects of $\Top^D$ and $\Top^{\mathcal{O}_{\mathcal{F}}^{op}}$.

    \definition[{\cite[Definition 1.2]{DF}}] Given a collection of orbits $\mathcal{F}$ of a small category $D$ and a $D$-space $X$, a \textit{relative $D$-CW structure of type $\mathcal{F}$ on $X$} is a sequence of $D$-spaces \[X^{-1} \hookrightarrow X^0 \hookrightarrow \ X^1 \hookrightarrow \dots \hookrightarrow X^n \hookrightarrow \dots \hookrightarrow X\] such that for each $i \geq 0 $, $X^i$ is obtained from $X^{i-1}$ as a pushout

    \[\begin{tikzcd}
    	{S^{i-1} \times A_i} && {X^{i-1}} \\
    	\\
    	{D^i \times A_i} && {X^i}
    	\arrow[hook, from=1-1, to=3-1]
    	\arrow[from=1-1, to=1-3]
    	\arrow[hook, from=1-3, to=3-3]
    	\arrow[from=3-1, to=3-3]
    	\arrow["\lrcorner"{anchor=center, pos=0.125, rotate=180}, draw=none, from=3-3, to=1-1]
    \end{tikzcd}\] where each $A_i$ is a disjoint union of $D$-orbits in $\mathcal{F}$. If $X^{-1}$ is the constant empty $D$-space, we drop the word relative and say that \[X^0 \hookrightarrow \ X^1 \hookrightarrow \dots \hookrightarrow X^n \hookrightarrow \dots \hookrightarrow X\] is a \textit{$D$-CW structure on $X$}.
    
As in the non-equivariant case, a map $\alpha: X \rightarrow Y$ of $D$-CW-complexes is a $D$-homotopy equivalence if an only if all $n$th homotopy group functors (including $n=0$) induce isomorphisms. This result is usually called ``Whitehead's theorem'' in the non-equivariant case and ``Bredon's theorem'' in the group-equivariant case. We now present the diagram-equivariant case, which was proven by Dror Farjoun and Zabrodsky.

    \definition[{\cite[Definition 2.2]{DF}}] Let $\mathcal{F}$ be a collection of orbits of a small category $D$. We say a $D$-space $X$ is \textit{of type $\mathcal{F}$} if \[\mathcal{O}_X = \{ O_x \mid x \in \Colim(X) \} \subseteq \mathcal{F}.\] 
    
    \thm[{\cite[Theorem 4.1]{DF}}] Let $D$ be a small category and let $\alpha: X \rightarrow Y$ be a $D$-equivariant map of $D$-CW-complexes of type $\mathcal{F}$. Then, $\alpha$ is a $D$-homotopy equivalence if and only if $\alpha^O:X^O \rightarrow Y^O$ is a homotopy equivalence of spaces for all $O \in \mathcal{F}$. (That is, that $\pi_n^O(\alpha):\pi_n^O(X) \rightarrow \pi_n^O(Y)$ is a isomorphism for all $n \in \mathcal{N}$ and $O \in \mathcal{F}$.)

In the previous section, we saw that $X = (S^m \hookrightarrow S^{\infty})$ and $Y = (S^m \rightarrow \{ \pt \})$ were not $\mathbb{J}$-homotopy equivalent. Since $X$ is of type $\mathcal{O}_X = \{ [0], [1] \}$ and $Y$ is of type $\mathcal{O}_Y = \{ Y \}$, ($Y$ is an orbit!) we know that any $\alpha:X \rightarrow Y$ must have $\pi_n^O(\alpha)$ fail to be an isomorphism for some $n \in \mathbb{N}$ and $O \in \{ [0], [1], Y \}$. Back then, we showed that $\pi_n^{[i]}(X)$ and $\pi_n^{[i]}(Y)$ were not isomorphic for $i \geq 2$ and any $n$ where  $\pi_n(S^m) \neq 0$. The theorem above says we could have simply checked $\pi_m^Y(X)$ and $\pi_m^Y(Y)$.

In general, once one has a $D$-CW structure on $X$, it's straightforward to know which orbits to check:

    \prop\label{CWType} If $X$ has (non-relative) $D$-CW structure of type $\mathcal{F}$, then $X$ is a $D$-space of type $\mathcal{F}$.
        \begin{proof}
            Let $X$ have a non-relative $D$-CW structure of type $\mathcal{F}$, and consider any point $x_d \in X_d$ for any object $d \in D$. We wish to show that $O_{x_d} \in \mathcal{F}$. By construction, $x_d$ is a point in the interior of $D^i \times A_i$ for exactly one $i \in \mathbb{N}$. By equivariance, each point in $O_{x_d}$ must also be a point in the interior of $D^i \times A_i$. Thus, the orbit type of $x_d$ in $X$ must be the same as its orbit type in $D^i \times A_i$. Because the latter is an orbit type in $\mathcal{F}$ and $x_d$ is a generic point, we're done.
        \end{proof}

Like in the group-equivariant case, $D$-CW-complexes are tame combinatorial objects that allow us to isolate certain homotopical behavior. However, sometimes we want to remove the restriction that higher-dimensional cells only attach onto lower-dimensional cells. When we get rid of this restriction, we get the more general notion of \textit{$D$-cell complexes}. $D$-cell complexes enable us to provide nice descriptions of the cofibrations in $\Top^D$ and $\Top^{\mathcal{O}_{\mathcal{F}}^{op}}$, which we will need to prove our generalization of Elmendorf's theorem.

    \definition Given a collection of orbits $\mathcal{F}$ of a small category $D$ and a $D$-space $X$, a \textit{$D$-cell structure of type $\mathcal{F}$ on $X$} is a (potentially transfinite) sequence of pushouts of the form:

    \[\begin{tikzcd}
    	{S^{n-1} \times O_{\mu}} && {X_{\mu}} \\
    	\\
    	{D^n \times O_{\mu}} && {X_{\mu+1}}
    	\arrow[from=1-1, to=1-3]
    	\arrow[from=1-3, to=3-3]
    	\arrow[from=1-1, to=3-1]
    	\arrow[from=3-1, to=3-3]
    \end{tikzcd}\] such that $\Colim(X_\mu) = X$, where each $O_{\mu} \in \mathcal{F}$,  and where $n \geq 0$ is allowed to vary with respect to $\mu$. (Here, the $(-1)$-sphere is the empty space.)

    That is, there is an ordinal $\lambda$ such that $X = \Colim_{\nu \leq \lambda}X_{\nu}$. For successor ordinals $\lambda = \mu+1$, $X_{\lambda}$ is obtained by the above pushout. For limit ordinals $\lambda$, $X_{\lambda} = \Colim_{\nu < \lambda}X_{\nu}$. If $X_0$ is the constant empty $D$-space, we drop the word relative and say $X$ is a \textit{$D$-cell complex of type $\mathcal{F}$}.
    
    As with $D$-CW-complexes, $D$-cell complexes of type $\mathcal{F}$ are spaces of type $\mathcal{F}$:
    
    \prop If $X$ is a (non-relative) $D$-cell complex of type $\mathcal{F}$, then $X$ is a $D$-space of type $\mathcal{F}$.
        \begin{proof}
            Take the proof of proposition \ref{CWType} and replace ``$D$-CW structure'' with ``$D$-cell structure.''
        \end{proof}
    
    \example Any relative $D$-CW-complex is a relative $D$-cell complex.
        \begin{proof}
            Let $X$ have a relative $D$-CW structure. By definition, each $A_i$ involved in the construction is a disjoint union of orbits $O_{\alpha}$. By the usual axioms of set theory, this collection of orbits can be well-ordered. We can then order the orbits of $A_0, A_1, \dots$ lexicographically and get a new well-ordered set of all the orbits involved.  This corresponds to some ordinal $\lambda$ which we will now use for labeling. For any orbit $O_{\mu}$, we build our attaching pushout as  
            \[\begin{tikzcd}
            	&& {} \\
            	&& {X^{n-1}} \\
            	{S^{n-1} \times O_{\mu}} && {X_{\mu}} \\
            	\\
            	{D^n \times O_{\mu}} && {X_{\mu+1}}
            	\arrow[from=3-1, to=3-3]
            	\arrow[from=3-3, to=5-3]
            	\arrow[from=3-1, to=5-1]
            	\arrow[from=5-1, to=5-3]
            	\arrow[hook, from=2-3, to=3-3]
            	\arrow[from=3-1, to=2-3]
            	\arrow["\lrcorner"{anchor=center, pos=0.125, rotate=180}, draw=none, from=5-3, to=3-1]
            \end{tikzcd}\] where the inclusion $X^n \hookrightarrow X_{\mu}$ is guaranteed by our lexicographic ordering and where the map $S^{n-1} \times O_{\mu} \rightarrow X^n$ is given by the disjoint union decomposition of $A_n$ (and corresponding decomposition of $S^{n-1} \times A_n$). $X_0$ is defined as $X^{-1}$.
        \end{proof}
    
From the last sentence of the proof, we are also able to conclude:

    \coro Any $D$-CW-complex is a $D$-cell complex.

\section{Model Structures on $\mathrm{Top}^D$}\label{section:models}
We can now establish a model structure on $\Top^D$ from that on $\Top$:

    \definition[{\cite{Quillen}}] The \textit{classical model structure} on $\Top$ is given by: 
        \begin{itemize}
            \item Weak equivalences are weak homotopy equivalences (that is, maps that induce isomorphisms for all $\pi_n$).
            \item Fibrations are ``Serre fibrations.''
            \item Cofibrations are retracts of relative cell complexes. 
        \end{itemize}
    As in any model category, a choice of two of $\{ \textrm{Fibrations}, \textrm{Cofibrations}, \textrm{Weak Equivalences} \}$ uniquely determines the third. It is thus a theorem that the weak equivalences and fibrations described above determine the cofibrations of the definition. 

From this, we can build a model structure on $\Top^C$ for any (possibly large) category, $C$. The model structure we're about to describe will most often be used on $C = \Top^{\mathcal{O}_{\mathcal{F}}^{op}}$, where $\mathcal{O}_{\mathcal{F}}$ is the full subcategory of $\Top^D$ whose objects are orbits $O \in \mathcal{F}$. We'll use a different model structure on $\Top^D$, which is why we're using the letter ``$C$'' here. 

    \definition \label{projmodel} Given a (possibly large) category $C$, the \textit{projective model structure on $\Top^C$} is given by the following conditions: 
        \begin{itemize}
            \item Weak equivalences $\beta:R \rightarrow S$ are such that $\beta_X$ is a weak homotopy equivalence for each object $X \in C$.
            \item Fibrations $\beta:R \rightarrow S$ are such that $\beta_X$ is a Serre fibration for each object $X \in C$.
            \item Cofibrations are retracts of relative $C$-cell complexes of type $\Free$, the collection of free orbits of $C$.
        \end{itemize}

The description of the cofibrations follows from \cite[Theorem 11.6.1]{Hirschhorn}. We now give our model structure on $\Top^D$:

    \definition \label{Fmodel} Let $D$ be a small category and let $\mathcal{F}$ be some collection of $D$-orbits that contains all of the free orbits. Then, the \textit{$\mathcal{F}$-model structure} on $\Top^D$ is given by:
        \begin{itemize}
            \item Weak equivalences $\alpha: X \rightarrow Y$ are such that $\Top^D(O,\alpha): \Top^D(O,X) \rightarrow \Top^D(O,Y)$ is a weak equivalence for each $O \in \mathcal{F}$.
            \item Fibrations $\alpha: X \rightarrow Y$ are such that $\Top^D(O,\alpha): \Top^D(O,X) \rightarrow \Top^D(O,Y)$ is a Serre fibration for each $O \in \mathcal{F}$.
        \end{itemize}

Our main theorem, which is proven in the next section, is that there is a Quillen equivalence,  
\[\begin{tikzcd}
	{\Top^D} \\
	\\
	{\Top^{\mathcal{O}_{\mathcal{F}}^{op}}}
	\arrow[""{name=0, anchor=center, inner sep=0}, "K", curve={height=-30pt}, from=3-1, to=1-1]
	\arrow[""{name=1, anchor=center, inner sep=0}, "\Phi", curve={height=-30pt}, from=1-1, to=3-1]
	\arrow["\dashv"{anchor=center}, draw=none, from=0, to=1]
\end{tikzcd}\]  where $\Top^D$ has the $\mathcal{F}$-model structure and $\Top^{\mathcal{O}_{\mathcal{F}}^{op}}$ has the projective model structure. The fact that this is a Quillen adjunction allows us to describe the cofibrations in the $\mathcal{F}$-model structure, which is the subject of Proposition \ref{FCofibs}

\section{Elmendorf's Theorem} \label{section:Elmendorf}

We now prove our diagram-equivariant version of Elmendorf's theorem. Our approach follows a modern treatment of the group-equivariant case by Marc Stephan \cite{Stephan}. There is a similar theorem for simplicial sets given by Dwyer and Kan \cite{Dwyer-Kan} that uses a different notion of ``orbit.'' We state our theorem as Theorem \ref{elmendorf} and spend the rest of this paper proving it.

\thm \label{elmendorf}
Let $D$ be a small category, let $\mathcal{F}$ be some collection of orbits of $D$ that contains all of the free orbits, and let $\mathcal{O}_{\mathcal{F}} \subseteq \Top^D$ be the full subcategory spanned by $\mathcal{F}$. Then, there is a Quillen equivalence \[K:\Top^{\mathcal{O}_{\mathcal{F}}^{op}} \simeq \Top^D: \Phi,\] where $\Top^{\mathcal{O}_{\mathcal{F}}^{op}}$ has the projective model structure and $\Top^D$ has the $\mathcal{F}$-model structure.

    \definition The functors that comprise the Quillen equivalence are: 
        \begin{itemize}
            \item $K:\Top^{\mathcal{O}_{\mathcal{F}}^{op}} \rightarrow \Top^D$ is defined by $K(R) = R \circ i$, where $i$ is the inclusion of $D$ into $\mathcal{O}_{\mathcal{F}}^{op}$ via the Yoneda embedding $d \mapsto F^D$. 
            \item $\Phi: \Top^D \rightarrow \Top^{\mathcal{O}_{\mathcal{F}}^{op}}$ is defined by $\Phi(X)(O) = \Top^D(O,X)$ for all $O \in \mathcal{F}$.
        \end{itemize}

To prove Theorem \ref{elmendorf}, we will show first show that $(K,\Phi)$ is an adjunction, then show that $(K,\Phi)$ is a Quillen adjunction, and then finally show that $(K,\Phi)$ is in fact a Quillen equivalence.

\begin{lemma}
$K$ is a left inverse to $\Phi$. That is, there is a natural isomorphism $K\Phi \cong \id_{\Top^D}$.
\end{lemma}

    \begin{proof}
    Let $X$ be a $D$-space. By definition, $K\Phi (X)$ is the $D$-space given by \[[K\Phi (X)]_d = \Top^D(F^d,X).\] But by the Yoneda lemma, $\Top^D(F^d,X) \cong X_d$. The naturality of the Yoneda lemma in the first argument thus tells us that that $K\Phi (X) \cong X$. The naturality of the Yoneda lemma in the second argument allows us to then conclude that $K\Phi \cong \id_{\Top^D}$.
    \end{proof}

\begin{prop} 
$(K,\Phi)$ is an adjunction.
\end{prop}

    \begin{proof}
    We will construct a natural isomorphism \[\Top^D(K(R),X) \cong \Top^{\mathcal{O}_{\mathcal{F}}^{op}}(R,\Phi (X)),\] where $R$ is a generic $\mathcal{O}_{\mathcal{F}}^{op}$-space and $X$ is a generic $D$-space. In this context, given a $D$-equivariant map ${f \colon K(R) \rightarrow X}$, its adjunct is the $\mathcal{O}_{\mathcal{F}}^{op}$-equivariant map $g: R \rightarrow \Phi (X)$ defined as follows:
    
    For any $O \in \mathcal{F}$, \[g(O): R(O) \rightarrow [\Phi (X)](O) = \Top^D(O,X)\] is the continuous map that sends $r \in R(O)$ to the $D$-equivariant map \[[g(O)](r):O \rightarrow X,\] where $[g(O)](r)$ is defined by \[([g(O)](r))_d(o_d) = f(R(o_d^*)(r))\] for all $d \in D$ and $o_d \in O_d$. (Here, $o_d^* \colon F^d \rightarrow O$ is the unique $D$-equivariant map that sends $\id_d$ to $o_d$.)
    
    For this construction to be valid, we need to check that $[g(O)](r)$ is indeed $D$-equivariant and then that $g$ is indeed $\mathcal{O}_{\mathcal{F}}^{op}$-equivariant. This just means that $[g(O)](r)$ and $g$ need to be natural transformations, so we check the corresponding naturality squares. We'll begin with $[g(O)](r)$:
    
    Let $\alpha : d \rightarrow d'$ be a morphism in $D$. To see that the square

    \[\begin{tikzcd}
    	{O_d} && {X_d} \\
    	\\
    	{O_{d'}} && {X_{d'}}
    	\arrow["{O_{\alpha}}"', from=1-1, to=3-1]
    	\arrow["{X_{\alpha}}", from=1-3, to=3-3]
    	\arrow["{([g(O)](r))_d}", from=1-1, to=1-3]
    	\arrow["{([g(O)](r))_{d'}}"', from=3-1, to=3-3]
    \end{tikzcd}\] commutes, consider a generic element $o_d \in O_d$ and observe that:
        \begin{enumerate}
            \item $[O_{\alpha}(o_d)]^* = o_d^* \circ \Top^D(\alpha, -)$ because the maps agree on $\id_{d'}$ and are $D$-equivariant.
            \item Applying $R$ to both sides gives us $R([O_{\alpha}(o_d)]^*) = R(\Top^D(\alpha, -)) \circ R(o_d^*)$. (R is contravariant!)
            \item By definition of $K$, $K(R)_d = R(F^d)$, $K(R)_{d'} = R(F^{d'})$, and $K(R)_{\alpha} = R(\Top^D(\alpha, -))$, so the previous line can be rephrased as $R([O_{\alpha}(o_d)]^*) = K(R)_{\alpha} \circ R(o_d^*)$.
            \item Since $f: K(R) \rightarrow X$ is a $D$-equivariant map, $f_{d'} \circ K(R)_{\alpha} = X_{\alpha} \circ f_d$. 
            \item Thus, combining the previous three steps, we see that \[f_{d'} \circ R([O_{\alpha}(h_d)]^*) = f_{d'} \circ R(\Top^D(\alpha, -)) \circ R(o_d^*) = X_{\alpha} \circ f_d \circ R(o_d^*).\]
            \item The above are all continuous maps from $R(O)$ to $X_{d'}$. Hence, for any $r \in R(O)$, $f_{d'} \circ R([O_{\alpha}(o_d)]^*)(r) =  X_{\alpha} \circ f_d \circ R(o_d^*) (r)$ as elements of $X_{d'}$.
            \item By definition of $g(O)(r)$, this shows that $[g(O)(r)]_{d'} \circ O_{\alpha} (o_d) = X_{\alpha} \circ [g(O)(r)]_d (o_d)$. Since $o_d$ was arbitrary, our square commutes and we conclude that $g(O)(r)$ is indeed $D$-equivariant.
        \end{enumerate}
        
    Now let's confirm that $g$ is $\mathcal{O}_{\mathcal{F}}^{op}$-equivariant, which is to say that the square 
    \[\begin{tikzcd}
    	{R(O)} &&& {\Phi(X)(O) = \Top^D(O,X)} \\
    	\\
    	{R(P)} &&& {\Phi(X)(P) = \Top^D(P,X)}
    	\arrow["{R(\sigma)}", from=3-1, to=1-1]
    	\arrow["{g(P)}"', from=3-1, to=3-4]
    	\arrow["{g(O)}", from=1-1, to=1-4]
    	\arrow["{\Top^D(\sigma,X) = (- \circ \sigma)}"', from=3-4, to=1-4]
    \end{tikzcd}\] commutes, where $\sigma: O \rightarrow P$ is any map of $D$-orbits. (Note the direction of the vertical arrows; $R$ and $\Phi(X)$ are contravariant.)
    
    Let $r$ be a generic element of $R(P)$. We can see $g(O)(R(\sigma)(r))$ and $g(P)(r) \circ \sigma$ are the same element of $\Top^D(O,X)$ by the following:
        \begin{enumerate}
            \item For any object $d \in D$ and point $o_d \in O_D$, we know that $\sigma \circ o_d^* = (\sigma(o_d))^*$ because they are both $D$-equivariant maps from $F^d$ that agree on $\id_d$.
            \item Thus, applying $R$ to both sides, we get that $R(o_d^*) \circ R(\sigma)$ and $R((\sigma(o_d))^*)$ are equal as functions from $R(P)$ to $R(F^d)$. Hence, for any $r \in R(P)$, $R(o_d^*) \circ R(\sigma)(r) = R((\sigma(o_d))^*)(r)$.
            \item Since $R(F^d) = K(R)_d$ by the definition of $K$, we can post-compose $f$ to both sides and see that \[f(R(o_d^*) \circ R(\sigma)(r)) = f(R((\sigma(o_d))^*)(r)).\]
            \item But by definition of $g$, the left-hand side of this equation is $g(O)(R(\sigma)(r))(o_d)$, and the right-hand side is $g(P)(r) \circ \sigma (o_d)$. Because this holds for all possible $r$ and $o_d$, $g$ is indeed $\mathcal{O}_{\mathcal{F}}^{op}$-equivariant.
        \end{enumerate}

    Having constructed the adjunct $g$, let us now show that the assignment of $f:K(R) \rightarrow X$ to $g:R \rightarrow \Phi(X)$ as described above yields a bijection $\Top^D(K(R),X) \cong \Top^{\mathcal{O}_{\mathcal{F}}^{op}}(R,\Phi (X))$. To see that the assignment is injective, observe that if $f_1, f_2: K(R) \rightarrow X$ differ at $r \in K(R)_d = R(F^d)$, then the corresponding $g_1$ and $g_2$ differ because $g_i(F^d)(r)(\id_d) = f_i(r)$.
    
    To show surjectivity, we will demonstrate that any $g:R \rightarrow \Phi(X)$ has an $f:K(R) \rightarrow X$ assigned to it, namely the one defined by \[f(r) = ([g(F^d)](r))_d(\id_d)\] for any $r \in R(F^d) = K(R)_d$.
    
    If we were to continue with the notation just used, the rest of the proof would be quite cumbersome. It's time to simplify:
    
        \notn From now on, $([g(O)](r))_d(o_d)$ will be denoted $g(O)(r)(o_d)$. (In particular, the domain of $g(O)(r)$, treated as a continuous map, will be implicit from the argument.)
    
    Resuming the proof of surjectivity, we first note that this $f$ is indeed $D$-equivariant because the naturality of $g$ in $F^d$ gives that $f$ is natural in $d$, and that $f_d$ is continuous because $g(F^d)$ is. To finish proving surjectivity, we just need to show that $f$ is actually assigned to $g$, which we do as follows:
    \begin{enumerate}
        \item The adjunct $\widehat{g}$ that $f$ is assigned to is defined by \[\widehat{g}(O)(r_O)(o_d) = f(R(o_d^*)(r_O))\] for any $r_O \in R(O)$. However, we've defined $f$ above such that \[f(R(o_d^*)(r_O)) = (g(F^d)(R(o_d^*)(r_O))(\id_d).\] 
        \item By naturality of $g$ in $O$, we know that $\widehat{g}(O)(r_O) \circ o_d^*$ and $g(F^d) \circ (R(o_d^*)(r_O))$ are equal as elements of $\Top^D(F^d,X)$. In particular, they agree on the evaluation of $\id_d$, which means \[\widehat{g}(O)(r_O)(o_d) = g(F^d)(R(o_d^*)(r_O))(\id_d) = g(O)(r)(o_d).\] Hence, $\widehat{g} = g$, so the generic assignment of $f$ to $g$ is surjective.
    \end{enumerate}

    Finally, to complete the proof of the adjunction, we just need to show that the bijection (isomorphism of sets) $\Top^D(K(R),X) \cong \Top^{\mathcal{O}_{\mathcal{F}}^{op}}(R,\Phi (X))$ is natural in both $X$ and $R$:
    \begin{enumerate} 
        \item Let $\alpha: X \rightarrow Y$ be a map of $D$-spaces. To show naturality  in $X$, we will check that the diagram 
        \[\begin{tikzcd}
        	{\Top^D(K(R),X)} && {\Top^{\mathcal{O}_{\mathcal{F}}^{op}}(R,\Phi(X))} \\
        	\\
        	{\Top^D(K(R),Y)} && {\Top^{\mathcal{O}_{\mathcal{F}}^{op}}(R,\Phi(Y))}
        	\arrow["\cong", from=1-1, to=1-3]
        	\arrow["\cong", from=3-1, to=3-3]
        	\arrow["{\alpha \circ -}"', from=1-1, to=3-1]
        	\arrow["{\Phi(\alpha) \circ -}"', from=1-3, to=3-3]
        \end{tikzcd}\] commutes. Consider any $f \in \Top^D(K(R),X)$, which thus has adjunct $g \in \Top^{\mathcal{O}_{\mathcal{F}}^{op}}(R,\Phi (X))$ defined by $g(O)(r)(h_d) = f(R(o_d^*)(r))$, for all $O \in \mathcal{O}_{\mathcal{F}}$, $r \in R(O)$, and $o_d \in O_d$. By definition of $\Phi$, the composition $\Phi(\alpha) \circ g$ thus satisfies $\Phi(\alpha) \circ g(O)(r)(o_d) = \alpha \circ f(R(o_d^*)(r))$. But $\alpha \circ f(R(o_d^*)(r))$ is precisely the formula that defines that adjunct to $\alpha \circ f \in \Top^D(K(R),Y)$. Hence, the diagram commutes, so $\Top^D(K(R),X) \cong \Top^{\mathcal{O}_{\mathcal{F}}^{op}}(R,\Phi (X))$ is natural in $X$.

        \item Similarly, let $\gamma:R \rightarrow S$ be a map of $\mathcal{O}_{\mathcal{F}}^{op}$-spaces. To show naturality in $R$, we will check that the diagram
        \[\begin{tikzcd}
        	{\Top^D(K(R),X)} && {\Top^{\mathcal{O}_{\mathcal{F}}^{op}}(R,\Phi(X))} \\
        	\\
        	{\Top^D(K(S),X)} && {\Top^{\mathcal{O}_{\mathcal{F}}^{op}}(S,\Phi(X))}
        	\arrow["\cong", from=1-1, to=1-3]
        	\arrow["\cong", from=3-1, to=3-3]
        	\arrow["{- \circ K(\gamma)}", from=3-1, to=1-1]
        	\arrow["{- \circ \gamma}", from=3-3, to=1-3]
        \end{tikzcd}\] commutes. (Note the direction of the vertical arrows.) We know that any $f \in \Top^D(K(S),X)$ is assigned to the adjunct $g \in \Top^{\mathcal{O}_{\mathcal{F}}^{op}}(S,\Phi (X))$ defined by $g(O)(s)(o_d) = f(S(o_d^*)(s))$ for all $O \in \mathcal{O}_{\mathcal{F}}$, $s \in S(O)$, and $o_d \in O_d$. The composition $g \circ \gamma$ then satisfies \[(g \circ \gamma) (O)(r)(o_d) = g(O)(\gamma(r))(o_d) = f(S(o_d^*)(\gamma(s)))f(\gamma \circ R(o_d^*)(\gamma(s))).\] (The first equation is the definition of $g \circ \gamma$, the second follows by plugging $\gamma(r)$ into the adjunct formula, and the third is given by the fact that $\gamma(R) = S$.) But $f(\gamma \circ R(o_d^*)(\gamma(s)))$ is precisely the formula that defines the adjunct of $f \circ K(\gamma)$. Hence, $\Top^D(K(R),X) \cong \Top^{\mathcal{O}_{\mathcal{F}}^{op}}(R,\Phi (X))$ is natural in $R$.
    \end{enumerate}
    
    Thus, we've completed the proof of that $(K,\Phi)$ is an adjunction.
    
    \end{proof}

\begin{lemma}
$(K,\Phi)$ is a Quillen adjunction.
\end{lemma}

    \begin{proof}
    One of the equivalent conditions for an adjunction to be a Quillen adjunction is that the right adjoint, $\Phi$, preserve fibrations and trivial fibrations (that is, fibrations that are also weak equivalences). Recall from Definitions \ref{projmodel} and \ref{Fmodel} that the model structures we're using are:
        \begin{itemize}
            \item The weak equivalences (or fibrations) in $\Top^D$ are maps $\beta: X \rightarrow Y$ such that $\Top^D(O,\beta):\Top^D(O,X) \rightarrow \Top^D(O,Y)$ is a weak equivalence (or fibration) in $\Top$ for all $O \in \mathcal{F}$.
            \item The weak equivalences (or fibrations) in $\Top^{\mathcal{O}_{\mathcal{F}}^{op}}$ are maps $\gamma: R \rightarrow S$ such that $\gamma(O):R(O) \rightarrow S(O)$ is a weak equivalence (or fibration) in $\Top$ for all $O \in \mathcal{F}$.
        \end{itemize}
     We observe from this description that, since $\Phi(X)(O) = \Top^D(O,X)$, a $D$-equivariant map $\beta: X \rightarrow Y$ is a weak equivalence (resp. fibration) if and only if $\Phi(\beta):\Phi(X) \rightarrow \Phi(Y)$ is an equivalence (resp. fibration). Hence, $\Phi$ preserves weak equivalences and fibrations, and thus also trivial fibrations.
    \end{proof}

We're now able to describe the cofibrations in the $\mathcal{F}$-model structure:

    \prop\label{FCofibs} Any relative cell complex $\alpha: X_0 \rightarrow X$ of type $\mathcal{F}$ is a cofibration in $\Top^D$ under the $\mathcal{F}$-model structure.
        \begin{proof}
            The left adjoint in a Quillen adjunction preserves cofibrations. It also preserves pushouts and general colimits. Thus, any $D$-cell complex $X$ of type $\mathcal{F}$ is the image under $K$ of a $\Top^{\mathcal{O}_{\mathcal{F}}^{op}}$-cell complex of type $\Free$. (A $D$-orbit in $\mathcal{F}$ is a free $\mathcal{O}_{\mathcal{F}}^{op}$-orbit under the Yoneda embedding. Thus, a $D$-cell complex where $D^n \times O_{\mu}$ is attached at the $\mu$th stage is hit by a $\Top^{\mathcal{O}_{\mathcal{F}}^{op}}$-cell complex where $\Top^{\mathcal{O}_{\mathcal{F}}^{op}}(O_{\mu},-) \times D^n$ is attached at the $\mu$th stage.) 
        \end{proof}

We can now finish the proof of the theorem:

\begin{thm}\label{elminduction} 
$(K, \Phi)$ is a Quillen equivalence.
\end{thm}

    \begin{proof}
    We need to show that, for any cofibrant $R \in \Top^{\mathcal{O}_{\mathcal{F}}^{op}}$ and fibrant $X \in \Top^D$, any $D$-equivariant map $f:K(R) \rightarrow X$ is a weak equivalence if and only if its adjunct $g:R \rightarrow \Phi(X)$ is a weak equivalence. But by how we've defined our weak equivalences, $f$ is a weak equivalence if and only if $\Phi(f):\Phi K (R) \rightarrow \Phi(X)$ is. By the 2-of-3 property of weak equivalences and that fact that the unit of the adjunction at $R$, $\eta_R$, factors $\Phi(f)$ as $f \circ \eta_R$, it is sufficient (and necessary) to show that $\eta_R$ is a weak equivalence for all cofibrant $R \in \Top^{\mathcal{O}_{\mathcal{F}}^{op}}$. We will in fact show that $\eta_R$ is an \textit{isomorphism} for each cofibrant $R$.
    
    Recall from the discussion in Definition \ref{projmodel} that cofibrations in the projective model structure on $\Top^{\mathcal{O}_{\mathcal{F}}^{op}}$ are retracts of relative $\Top^{\mathcal{O}_{\mathcal{F}}^{op}}$-cell complexes of type $\Free$. Hence, every cofibrant $R$ can be realized as a retract of some $R'$, where $R'$ is a transfinite composition of pushouts of the form 
    \[\begin{tikzcd}
    	{\Top^{\mathcal{O}_{\mathcal{F}}^{op}}(O,-) \times A} && {\widetilde{R}_{\mu}} \\
    	\\
    	{\Top^{\mathcal{O}_{\mathcal{F}}^{op}}(O,-) \times B} && {\widetilde{R}_{\mu+1}}
    	\arrow["{\id_{\Top^{\mathcal{O}_{\mathcal{F}}^{op}}(O,-)} \times c}"', from=1-1, to=3-1]
    	\arrow[from=1-1, to=1-3]
    	\arrow[from=3-1, to=3-3]
    	\arrow[from=1-3, to=3-3]
    	\arrow["\lrcorner"{anchor=center, pos=0.125, rotate=180}, draw=none, from=3-3, to=1-1]
    \end{tikzcd}.\] Since $R $ is a retract of $R'$, we can show $\eta_R$ is an isomorphism by showing that $\eta_{R'}$ is an isomorphism. We will do this by transfinite induction: 
    
    Let $\lambda$ be an ordinal such that there is a functor $\widetilde{R}: \lambda \rightarrow \Top^{\mathcal{O}_{\mathcal{F}}^{op}}$ with colimit $R'$ and such that for all successor ordinals $\mu +1 < \lambda$, $\widetilde{R}_{\mu + 1}$ is the pushout given above. ($O$ is some orbit in  $\mathcal{O}_{\mathcal{F}}$ and $c: A \rightarrow B$ is a generating cofibration in $\Top$.
    
    \underline{Initial Case:} If $\lambda$ is the initial ordinal, then the colimit of $\lambda$ is the initial object of $\Top^{\mathcal{O}_{\mathcal{F}}^{op}}$. In this case, $R'(O)=\{ \}$ for all $O \in \mathcal{F}$. Hence, $K(R')_d = \{ \}$ for all objects $d \in D$, and consequently $\Phi K (R') (O) = \Top^D(O,K(R')) = \{ \}$ as all orbits have at least one point and there are no continuous maps from a non-empty to the empty space. This makes$\eta_{R'}$ an equality and in particular an isomorphism.
    
    \underline{Successor Case:} If $\lambda = \mu +1$ for some ordinal $\mu$, we inductively assume that $\eta_{\widetilde{R}_{\mu}}$ is an isomorphism. To see that $\eta_{R'}$ is an isomorphism, we will check that $\Phi K (-)$ preserves pushouts and that $\eta_{\Top^{\mathcal{O}_{\mathcal{F}}^{op}}(O,-) \times A}$ is an isomorphism for all $O \in \mathcal{F}$ and $A \in \Top$. For pushout-preservation, we first note that $K$ automatically preserves pushouts by being a left adjoint. Since colimits (and limits) in a functor category are computed object-wise, showing that $\Phi$ preserves pushouts is equivalent to showing that $\Top^D(O, Y \sqcup_X Z)$ is naturally isomorphic to $\Top^D(O,Y) \sqcup_{\Top^D(O,X)} \Top^D(O,Z)$. But this is immediate from Proposition \ref{OfixedPushout}, since pushouts are colimits.\\
    
    To complete the successor ordinal case, we just need to show that $\eta_{\Top^{\mathcal{O}_{\mathcal{F}}^{op}}(O,-) \times A}$ is an isomorphism for all $O \in \mathcal{F}$ and $A \in \Top$. To see this, note that \[K(\Top^{\mathcal{O}_{\mathcal{F}}^{op}}(O,-) \times A)_d = \Top^D(O,F^d) \times A \cong O_d \times A.\] The naturality of the Yoneda lemma in the second argument shows us that \[K(\Top^{\mathcal{O}_{\mathcal{F}}^{op}}(O,-) \times A) \cong O \times A.\] Applying $\Phi$ to both sides gives \[\Phi K(\Top^{\mathcal{O}_{\mathcal{F}}^{op}}(O,-) \times A) \cong \Phi (O \times A).\] Next, observe that, at any orbit $P \in \mathcal{F}$, \[\Phi (O \times A)(P) = \Top^D(P, O \times A) \cong \Top^D(P,O) \times A = \Top^{\mathcal{O}_{\mathcal{F}}^{op}}(O,P) \times A.\] (To see the natural isomorphism above, first observe that \[\Top^D(P, O \times A) \cong \Top^D(P,O) \times \Top^D(P,A)\] because representable functors preserve limits. Then, simplify by noting that $\Top^D(P,A) \cong A$ because $\textrm{colim}(P) = \{p\}$ is a one-point space and every morphism in $A$ is an identity morphism.) Since $\Top^D(P,O) \times A$  is the same space as $\Top^{\mathcal{O}_{\mathcal{F}}^{op}}(O,-) \times A$ evaluated at $P$, we conclude that $\Phi K(\Top^{\mathcal{O}_{\mathcal{F}}^{op}}(O,-) \times A)$ is naturally isomorphic to $\Top^{\mathcal{O}_{\mathcal{F}}^{op}}(O,-) \times A$. By construction, the natural isomorphism described above is precisely $\eta_{\Top^{\mathcal{O}_{\mathcal{F}}^{op}}(O,-) \times A}$, meaning the successor case is complete. 

    \underline{Limit Case}: Let $\lambda$ be a limit ordinal. We inductively assume that $\eta_{\widetilde{R}_{\nu}}$ is an isomorphism for all $\nu < \lambda$. Since $\lambda$ is a limit ordinal, $R' = \Colim_{\nu < \lambda}(\widetilde{R}_{\nu})$. As before, showing that $\eta_{R'}:\Phi K (R') \rightarrow R'$ is an isomorphism reduces to showing each $\eta_{R'}(O):\Phi K (R')(O) \rightarrow R'(O)$ is an isomorphism of spaces. Note that each $K(\widetilde{R}_{\nu})$ is a $D$-cell complex by construction, (Specifically, it's only built out of cells of orbit types $O \in \mathcal{F}$.) and $K(\widetilde{R}_{\nu}) \rightarrow K(\widetilde{R}_{\xi})$ is a $D$-cellular inclusion for any $\nu < \xi < \lambda$. Thus, since $K$ preserves colimits, we only need to check that $\Top^D(O,-)$ preserves colimits indexed by ordinals where each map is a $D$-cellular inclusion of $D$-cell complexes. Combined with the inductive hypothesis, this will show that $\eta_{R'}(O):\Phi K (R')(O) \rightarrow R'(O)$ is an isomorphism of spaces. 
    
    Let $f:O \rightarrow R' = \Colim_{\nu < \lambda}(\widetilde{R}_{\nu})$ be a $D$-equivariant map. We wish to find an ordinal $\nu < \lambda$ and map $g:O \rightarrow \widetilde{R}_{\nu}$ such that $g$ factors $f$. Let $o_d \in O_d$ be a point of $O$, and set $\nu$ to be the smallest ordinal such that $f(o_d) \in (\widetilde{R}_{\nu})_d$. Observe that when $f(o_d)$ was added to $\widetilde{R}_{\nu}$, it was added as the interior of some $D$-disk $Orb_{\mathcal{F}}^{op}(P,F^d) \times D^{n+1} \cong P \times D^{n+1}$. By equivariance and the colimit property of orbits, all points in $\widetilde{R}_{\nu}$ that are in the orbit of $f(o_d)$ must also have been in the interior of the new $(P \times D^{n+1})$-cell. Thus, all points in the orbit of $f(o_d)$ must lie in $\widetilde{R}_{\nu}$, which means that $f$ is indeed factored by a map $g:O \rightarrow \widetilde{R}_{\nu}$. This completes the limit case and the proof. 
\end{proof}

\bibliographystyle{alpha}
\bibliography{cited}

\end{document}